\documentclass[a4paper,11pt]{article}
\usepackage{authblk,setspace}

\usepackage{multicol} 
\usepackage{soul} 
\usepackage{cancel}
\usepackage{dsfont,bbm}
\usepackage{empheq,cases}
\usepackage[export]{adjustbox}

\DeclareFontFamily{OT1}{pzc}{}
\DeclareFontShape{OT1}{pzc}{m}{it}{<-> s * [1.10] pzcmi7t}{}
\DeclareMathAlphabet{\mathpzc}{OT1}{pzc}{m}{it}

\usepackage[british]{babel}
\usepackage{csquotes}
\usepackage{calligra}
\DeclareMathAlphabet{\mathcalligra}{T1}{calligra}{m}{n}

\usepackage[low-sup]{subdepth}

\let\baraccent=\= 
\renewcommand{\=}[1]{\stackrel{#1}{=}} 

\newcommand{\bs}[1]{\ensuremath{\boldsymbol{#1}}}

\usepackage{graphicx}
\usepackage{amsmath,amsthm,amssymb,bm,color,comment}
\usepackage{mathrsfs,enumerate}
\usepackage{subfig}
\usepackage{caption}
\usepackage[font=small,format=plain,labelfont=bf,textfont=it]{caption}
\usepackage{url}
\numberwithin{equation}{section}
\definecolor{refkey}{rgb}{0.9451,0.2706,0.4941}
\definecolor{labelkey}{rgb}{0.9451,0.2706,0.4941}
\newtheorem{theorem}{Theorem}[section]
\newtheorem*{theorem*}{Theorem}

\theoremstyle{definition}

\newtheoremstyle{myremstyle}
  {\topsep} 
  {1.2\topsep} 
  {} 
  {} 
  {\itshape} 
  {.} 
  {.5em} 
  {} 

\theoremstyle{myremstyle} 

\usepackage{fancybox,here,mathtools}

\usepackage[
	style=numeric,
	backend=bibtex,
	backref=true,
	sorting=nyt,
	maxbibnames=99,
	giveninits=true,
	doi=false,
	isbn=false
	]{biblatex}

\DefineBibliographyStrings{english}{%
  backrefpage = {page},
  backrefpages = {pages},
}

\renewbibmacro{in:}{}

\DeclareFieldFormat{postnote}{#1}
\DeclareFieldFormat{multipostnote}{#1}
	
\AtEveryBibitem{
\clearfield{month}
\clearfield{number}
\clearlist{location}
}
\AtEveryBibitem{\clearlist{language}} 
\DeclareFieldFormat[article,unpublished]{title}{#1}
\AtEveryBibitem{%
  \ifentrytype{book}{%
    \clearfield{pages}%
  }{%
  }%
}
\AtEveryBibitem{%
\ifentrytype{book}{
    \clearfield{url}%
}{}
\ifentrytype{article}{
    \clearfield{url}%
}{}
\ifentrytype{incollection}{
    \clearfield{url}%
}{}
}

\newcommand{\Addressesone}{{
		\bigskip
		\footnotesize
		\textsc{Y. Kazashi}\par\nopagebreak
		\textsc{Mathematics Institute, CSQI}\par\nopagebreak
		\textsc{\'{E}cole Polytechnique F\'{e}d\'{e}rale de Lausanne}
		\textsc{University of New South Wales}\par\nopagebreak
		\textsc{CH-1015 Lausanne, Switzerland}\vspace{1pt}\par\nopagebreak
		\textit{E-mail}: \texttt{\href{mailto:yoshihito.kazashi@epfl.ch}{yoshihito.kazashi@epfl.ch}}
}}
\newcommand{\Addressestwo}{{
	\bigskip
  \footnotesize
  \textsc{F.\ Y. Kuo}\par\nopagebreak
  \textsc{School of Mathematics and Statistics}\par\nopagebreak
  \textsc{University of New South Wales}\par\nopagebreak
  \textsc{Sydney NSW 2052, Australia}\vspace{1pt}\par\nopagebreak
  \textit{E-mail}: \texttt{\href{mailto:f.kuo@unsw.edu.au}{f.kuo@unsw.edu.au}}
}
}

\newcommand{\Addressesthree}{{
		\bigskip
		\footnotesize
		\textsc{I.\ H. Sloan}\par\nopagebreak
		\textsc{School of Mathematics and Statistics}\par\nopagebreak
		\textsc{University of New South Wales}\par\nopagebreak
		\textsc{Sydney NSW 2052, Australia}\vspace{1pt}\par\nopagebreak
		\textit{E-mail}: \texttt{\href{mailto:i.sloan@unsw.edu.au}{i.sloan@unsw.edu.au}}
}}

\author{Yoshihito Kazashi, Frances Y.\ Kuo, Ian H.\ Sloan}
\bibliography{mybib_CTAC}
\title{Derandomised lattice rules for high dimensional integration}

\date{}

\usepackage{algorithm,algcompatible}
\algnewcommand\INPUT{\item[\textbf{Input:}]}%
\algnewcommand\OUTPUT{\item[\textbf{Output:}]}%

\usepackage{amssymb}    

\theoremstyle{definition}
\usepackage{color}
\definecolor{refkey}{rgb}{0.9451,0.2706,0.4941}\definecolor{labelkey}{rgb}{0.9451,0.2706,0.4941}

\definecolor{darkred}{RGB}{139,0,0}
\definecolor{darkgreen}{RGB}{0,100,0}
\definecolor{darkmagenta}{RGB}{139,0,139}

\newcommand{\setu}{{\mathrm{\mathfrak{u}}}}

\newcommand{\bsx}{{\boldsymbol{x}}}

\newcommand{\bsz}{{\boldsymbol{z}}}

\newcommand{\bsDelta}{\boldsymbol{\Delta}}
\newcommand{\bsgamma}{{\boldsymbol{\gamma}}}
\newcommand{\rd}{{\mathrm{d}}}

\begin{document}
	
	\maketitle
	
	\begin{abstract}
		We seek shifted lattice rules that are good for high dimensional
		integration over the unit cube in the setting of an unanchored weighted
		{Sobolev} space of functions with square-integrable mixed first
		derivatives. Many existing studies rely on random shifting of the lattice,
		whereas here we work with lattice rules with a deterministic shift.
		Specifically, we consider ``half-shifted'' rules, in which each component
		of the shift is an odd multiple of $1/(2N)$, where $N$ is the number of
		points in the lattice. We show, by applying the principle that \emph{there
			is always at least one choice as good as the average}, that for a given
		generating vector there exists a half-shifted rule whose squared
		worst-case error differs from the shift-averaged squared worst-case error
		by a term of order only ${1/N^2}$. Numerical experiments, in which the
		generating vector is chosen component-by-component (CBC) as for randomly
		shifted lattices and then the shift by a new ``CBC for shift'' algorithm,
		yield encouraging results.
	\end{abstract}
	
	
	
	\section{Introduction}
	Lattice rules are often used for high dimensional integration over the
	unit cube, that is, for the numerical evaluation of the $s$-dimensional
	integral
	\begin{equation*}
		I_s(f)
		:= \int_0^1 \cdots \int_0^1 f(x_1,\ldots, x_{s}) \,\rd x_1 \cdots\rd x_{s}
		=\int_{[0,1]^s} f(\bsx)\,\rd \bsx.
	\end{equation*}
	A \emph{shifted lattice rule} for the approximation of the integral is an
	equal weight cubature rule of the form
	\begin{equation}\label{eq:shifted}
	Q_{N,s}(\bsz,\bsDelta;f)
	:= \frac{1}{N} \sum_{k=1}^{N} f\bigg(\bigg\{\frac{{k}\bsz}{N}+\bsDelta\bigg\}\bigg),
	\end{equation}
	where $\bsz\in\{1,\ldots,N-1\}^{{s}}$ is the \emph{generating vector},
	$\bsDelta\in[0,1]^s$ is the \emph{shift}, while the braces around an
	$s$-vector indicate that each component of the vector is to be replaced by
	its fractional part in $[0,1)$. The special case $\bsDelta = \bs{0}$
	yields the unshifted lattice rule, which provably works well for periodic
	functions \cite{Sloan.I.H_Joe_1994_book}. If the integrand is not
	periodic, the shift can play a useful role. The implementation of a
	shifted lattice rule is relatively easy once the vectors $\bsz$ and
	$\bsDelta$ are prescribed, even when $s$ is very large, say, in the tens
	of thousands.
	
	The central concern of this paper is the construction of a good shift
	vector~$\bsDelta$, given a specific choice of a good $\bsz$.

	At the present time the overwhelmingly favored method for fixing the
	shifts in a non-periodic setting is to choose them \emph{randomly}. In a
	\emph{randomly shifted lattice rule} the shift $\bsDelta$ is chosen from a
	uniform distribution on $[0,1]^s$, and the integral is approximated by an
	empirical estimate of the expected value, $\frac{1}{q}\sum_{i=1}^q
	Q_{N,s}(\bsz,\bsDelta_i;f)$, where $q$ is some fixed number, and
	$\bsDelta_1,\ldots, \bsDelta_q\in [0,1]^s$ are $q$ independent samples
	from the uniform distribution on ${[0,1]^s}$. With the shift chosen
	randomly, all that remains in the randomly shifted case is to construct
	the integer vector $\bsz$, which can be done very effectively by using the
	\emph{component-by-component} {(CBC)} construction to yield a vector
	{$\bsz^*$} that gives a satisfactorily small value of the
	\emph{shift-averaged worst-case error}.
	
	In the present paper we construct a new kind of shifted lattice rule,
	which is \emph{derandomised} in the sense that the generating vector is
	the same $\bsz^*$ determined by the CBC algorithm for the shift-averaged
	worst-case error, while the shift $\bsDelta^*$ is determined by a new CBC
	construction, ``\emph{CBC for shift}'': the components of the shift vector
	is obtained one at a time, chosen from the \emph{odd multiples of
		$1/(2N)$}.
	
	We argue that there is a significant potential cost saving in this
	deterministic alternative, in that it becomes no longer necessary to
	compute an empirical average over shifts. In many applications there is
	not just a single integral to be evaluated, but rather many such integrals
	with different input parameters.  In such a situation it may be seen as
	overkill to compute an error estimate for every single integral.  If at
	any stage an error estimate is needed by a user of the present algorithm,
	then all that is needed is to replace the computed shift $\bsDelta^*$ by
	$q$ randomly chosen shifts, in the knowledge that the resulting error
	estimate is just that of the CBC-constructed randomly shifted lattice
	rule.
	
	Approaches to estimating the error for lattice rules for non-periodic
	functions without randomisation include
	\cite{Dick.J_etal_2014_tent,Goda.T_SY_2017_tent}, where a mapping called
	the {\emph{tent transform}} was applied to the lattice rule. In this
	paper, however, no transformation of the lattice points is considered.
	
	\subsection{Function spaces and worst-case errors}
	
	The central element in any CBC construction is the \emph{worst-case
		error}, which for the case of the  shifted lattice rule~\eqref{eq:shifted}
	and a Hilbert space $H_s$ may be defined by
	\[
	e_{N,s}(\bsz,\bsDelta)
	:= {\sup_{f\in H_s, \|f\|_{H_s} \le 1}} |Q_{N,s}(\bsz,\bsDelta;f)-I_s(f)|.
	\]
	Here we consider a weighted \emph{unanchored} {Sobolev} space of functions
	with square-integrable mixed first derivatives on $(0,1)^s$, with squared
	norm
	\begin{align*}
		& \|f\|_{H_{s,\bsgamma}}^2
		\,:=\,
		&
		\sum_{\setu\subseteq\{1:s\}}
		\gamma_\setu^{-1}
		\int_{[0,1]^{|\setu|}}
		\left(\int_{[0,1]^{s-|\setu|}}
		\frac{\partial^{|\setu|}f}{\partial \bsx_\setu}(\bsx_\setu;\bsx_{\{1:s\}\setminus\setu})
		\,\rd\bsx_{\{1:s\}\setminus\setu}
		\right)^2
		\rd\bsx_\setu,
	\end{align*}
	where $\{1:s\} = \{1,2,\ldots,s\}$, $\bsgamma_\setu$ is a positive number
	{which} is the ``weight'' corresponding to the subset $\setu\subseteq
	\{1:s\}$, with $\bsgamma_\emptyset = 1$, and $\bsx_\setu$ denotes the
	variables $x_j$ for $j\in \setu$. It is well known that suitably decaying
	weights are essential if we are to have error bounds independent of
	dimension \cite{Sloan.I.H_Wozniakowski_1998_iff}. The squared worst-case
	error has an explicit formula (see e.g., \cite{Dick.J_etal_2013_Acta})
	\begin{equation}\label{eq:wce-formula}
	e_{N,s}^{2}(\bsz,\bsDelta)
	= {\frac{1}{N^{2}}\sum_{k=1}^{N} \sum_{k'=1}^{N}
		\sum_{{\emptyset \neq } \setu \subseteq \{1:s\}} \!\!\gamma_{\setu}\prod_{j\in \setu}
		\bigg[
		\frac{1}{2} B_{2}\bigg(\bigg\{\frac{(k-k')z_{j}}{N}\bigg\} \bigg)
		+ A_{k,k',z_j}(\Delta_j) \bigg],}
	\end{equation}
	where
	\[
	A_{k,k',z}(\Delta) :=
	\bigg(\bigg\{\frac{kz}{N}+\Delta\bigg\}-\frac{1}{2}\bigg)
	\bigg(\bigg\{\frac{k'z}{N}+\Delta\bigg\}-\frac{1}{2}\bigg).
	\]
	
	For the randomly shifted lattice rule the relevant form of the worst-case
	error is the \emph{shift-averaged worst-case error} $e_{N,s}^{{\rm
			sh}}(\bsz)$, defined by
	\begin{align} \label{eq:e-sh-sq}
		{[e_{N,s}^{{\rm sh}}(\bsz)]^2}
		:=\, \int_{[0,1]^s}e_{N,s}^{2}(\bsz,\bsDelta) \,\rd\bsDelta
		=\, {\frac{1}{N}\sum_{k=1}^{N}
			\sum_{\emptyset \neq \setu \subseteq \{1:s\}} \!\!\gamma_{\setu}\prod_{j\in \setu}
			B_{2}\bigg(\bigg\{\frac{kz_{j}}{N}\bigg\}\bigg)},
	\end{align}
	which is the expected value of the squared worst-case error, taken with
	respect to the random shift. Notice that the double sum over $k,k'$ in
	\eqref{eq:wce-formula} simplified to a single sum over $k$ in
	\eqref{eq:e-sh-sq}.
	\subsection{CBC constructions}
	The principle of a CBC construction is that{,} at stage $j$, one
	determines the $j$th component of the cubature points by seeking to
	minimise {an error criterion} for the $j$-dimensional problem; then with
	that component fixed one moves on to the next component, never going back.
	
	In the case of randomly shifted lattice rules, $z_1^*$ is chosen to have
	the value $1$, and then, for $j=1, 2, \ldots, s-1$, once
	$z_1^*,z_2^*,\ldots, z_j^*$ are fixed, {$z_{j+1}$} is chosen to be the
	element from $\{1, \ldots, N-1\}$ that gives the smallest value of
	$[e_{N,j+1}^{{\rm sh}}(z_1^*, \ldots, z_j^*, z_{j+1})]^2$. The cost of the
	CBC algorithm for constructing $\bsz^*$ up to $s$ dimensions is of order
	$s\,N\log N$ using FFT {\cite{Nuyens.D_Cools_2006_FFT_prime}}, in the
	simplest case of ``product weights'', in which there is only one sequence
	$\gamma_1,\gamma_2,\ldots,\gamma_s$ of weight parameters, and the value of
	$\gamma_\setu$ is taken to be the product $\prod_{j\in\setu} \gamma_j$. In
	this case the sum over $\setu$ in \eqref{eq:e-sh-sq} can be rewritten as a
	product of $s$ factors.
	
	The proven quality of the CBC construction for randomly shifted lattice
	rules is very good, in the sense that, with $\zeta$ being the Riemann zeta
	function and $\varphi$ being the Euler totient function, for all
	$\lambda\in (\frac{1}{2},1]$,
	\begin{equation}\label{eq:error_kuo}
	e_{N,s}^{{\rm sh}}(\bsz^*) \le
	\Bigg(
	{\frac{1}{\varphi(N)}\sum_{\emptyset\ne\setu\subseteq\{1:s\}}
		\gamma_\setu^\lambda \left(\frac{2\zeta(2\lambda)}{(2\pi^2)^\lambda}\right)^{|\setu|}}
	\Bigg)^{1/(2\lambda)},
	\end{equation}
	see e.g., \cite{Dick.J_etal_2013_Acta}. It follows from the definition
	that for $f\in H_{{s,\bsgamma}}$ one has as an  error bound for the
	{randomly shifted} lattice rule constructed by CBC
	\begin{equation*}
		\sqrt{\mathbb{E}[|Q_{N,s}(\bsz^*{,\cdot;}f)-I_s(f)|^2]}
		\le \!\Bigg(
		{\frac{1}{\varphi(N)}\sum_{\emptyset\ne\setu\subseteq\{1:s\}}
			\!\!\gamma_\setu^\lambda \left(\frac{2\zeta(2\lambda)}{(2\pi^2)^\lambda}\right)^{|\setu|}}
		\Bigg)^{1/(2\lambda)}
		\!\!\|f\|_{H_{{s,\bsgamma}}}.
	\end{equation*}
	When $N$ is prime we have $\varphi(N) = N-1$. Thus the convergence rate is
	arbitrarily close to $1/N$ as $\lambda\to 1/2$, but with a constant that
	blows up as $\lambda\to 1/2$ because $\zeta(2\lambda)\to \infty$.
	
	For our new derandomised lattice rule we take the components of the
	generating vector to be $z_1^*, z_2^*,\ldots, z_s^*$ as determined by the
	CBC algorithm for randomly shifted lattice rules.  We then determine the
	components of the shift by a new \emph{CBC for {shift}} algorithm, in
	which at stage $j$, with $\Delta_1^*, \ldots, \Delta_j^*$ already fixed,
	we choose $\Delta_{j+1}$ by minimising the {squared worst-case error
		$e_{N,j+1}^2((z_1^*, \ldots, z_j^*,z_{j+1}^*),(\Delta_1^*, \ldots,
		\Delta_j^*,\Delta_{j+1}))$}.  Of course it is not possible to check all
	real numbers in $[0,1)$ as values of $\Delta_1, \ldots, {\Delta_s}$. We
	argue that it is sufficient to restrict the set of possible shift
	components to the odd multiples of {$1/(2N)$}, that is, to the $N$
	values $S_N := \{1/(2N), 3/(2N), \ldots, (2N-1)/(2N)\}$.
	
	Our argument for the sufficiency of restricting the search over shifts to
	the odd multiples of $1/(2N)$ is given in Theorem~\ref{thm} below, in
	which we show that for any choice of generating vector $\bsz$, the average
	of the squared worst-case error over all shifts in $[0,1]^s$ differs from
	the average over the discrete set $S_N^s$ by a term of order only $1/N^2$.
	
	The restriction from the continuous interval $[0,1]$ to the discrete set
	$S_N$ for the shift has previously been considered in
	\cite{Sloan.I.H_KJ_2002a_step} in a different CBC algorithm which
	constructs the components of $\bsz$ and $\bsDelta$ simultaneously, in the
	order of $z_1, \Delta_1,z_2,\Delta_2,\ldots$.
	
	Now we discuss the error with respect to the shift $\bsDelta^*$ obtained
	by the present CBC for shift algorithm. Let us define the ratio:
	\begin{equation}
	\kappa(N,s):=\frac{e_{N,s}(\bsz^*,\bsDelta^*)}{{e_{N,s}^{\rm{sh}}(\bsz^*)}}.
	\label{eq:ratio}
	\end{equation}
	Then, from the definition of the worst-case error {and} using
	\eqref{eq:error_kuo} we have the the following error bound for the present
	CBC algorithm:
	\begin{align*}
		&|Q_{N,s}(\bsz^*,\bsDelta^*;f)-I_s(f)|
		\le \kappa(N,s)\, e_{N,s}^{{\rm sh}}(\bsz^*)\,\|f\|_{H_{{s,\bsgamma}}}\\
		&\qquad\qquad\qquad\le \kappa(N,s)
		\Bigg(
		{\frac{1}{\varphi(N)}\sum_{\emptyset\ne\setu\subseteq\{1:s\}}
			\!\!\gamma_\setu^\lambda
			\bigg(
			\frac{2\zeta(2\lambda)}{(2\pi^2)^\lambda}
			\bigg)^{|\setu|}}
		\Bigg)^{1/(2\lambda)}
		\!\!\|f\|_{H_{{s,\bsgamma}}},
	\end{align*}
	for all $\lambda\in (1/2,1]$, which is an \emph{explicit} and
	\emph{deterministic} error bound in which in any practical situation
	$\kappa(N,s)$ is a known constant. Numerical experiments in
	Section~\ref{sec:CBC-for-shifts} suggest that $\kappa(N,s)$ can often be
	smaller than $1$, making the derandomised option attractive in practice.
	
	It should be said the presented CBC for shift algorithm is expensive: the
	cost of a single evaluation of the worst-case error \eqref{eq:wce-formula}
	is of order $sN^2$ in the simplest case of product weights, and therefore
	the cost of a search over $N$ values of the shift up to dimension $s$ is
	of order $sN^3$ (if we store the products during the search). But the cost
	is an off-line cost, since spare computing capacity can be used to
	complement existing CBC vectors $\bsz^*$ for randomly shifted lattice
	rules by deterministic shifts $\bsDelta^*$ generated by the CBC for shift
	algorithm.
	
	\section{Error analysis}
	
	In this section, we show that for any choice of generating vector $\bsz$,
	the squared worst-case error with shift averaged over {$S_N^s$},
	defined by
	\begin{equation}
	[e_{N,s}^{\frac{1}{2}\rm{sh}}(\boldsymbol{z})]^2 :=\frac{1}{N^{s}}
	\sum _{\boldsymbol{\Delta}\in S_N^s}
	e^{2}_{N,s}(\boldsymbol{z} ; \boldsymbol{\Delta}),
	\label{eq:e-half-sh-sq}
	\end{equation}
	differs from the average of the squared worst-case error over all shifts
	$[e^{\rm{sh}}_{N,s}(\boldsymbol{z})]^2$ by a term of order only ${1/N^2}$.
	
	\begin{theorem}\label{thm} For
		arbitrary $\bsz\in \{1,\dotsc ,N-1\}^{s}$, with
		$e_{N,s}^{\rm{sh}}(\boldsymbol{z})$ and
		$e_{N,s}^{\frac{1}{2}\rm{sh}}(\bsz)$ as defined in \eqref{eq:e-sh-sq} and
		\eqref{eq:e-half-sh-sq}, respectively, we have
		\begin{equation*}
			\Big|
			[e_{N,s}^{\rm{sh}}(\boldsymbol{z})]^2
			-
			[e_{N,s}^{\frac{1}{2}\rm{sh}}(\boldsymbol{z})]^2
			\Big|
			\leq
			\frac{1}{4N^2}
			\sum_{\emptyset\neq \setu\subseteq\{1:s\}}\gamma_\setu\bigg(\frac{1}{3}\bigg)^{|\mathfrak{u} |} {|\setu|}.
		\end{equation*}
	\end{theorem}
	
	\begin{proof}
		We see from \eqref{eq:wce-formula} that
		\begin{align*}
			[e_{N,s}^{\rm sh}(\boldsymbol{z})]^2
			-
			[e_{N,s}^{\frac{1}{2}\rm{sh}}(\boldsymbol{z})]^2
			= \frac{1}{N^{2}}\sum_{k=1}^{N} \sum_{k'=1}^{N}
			\sum_{\emptyset \ne \setu \subseteq \{1:s\}} \gamma_{\setu}
			\left(\prod_{j\in\setu} a_j^{k,k'} - \prod_{j\in\setu} b_j^{k,k'}\right),
		\end{align*}
		where we write for $k,k'=1,\ldots,N$, $j=1,\ldots,s$, $m=1,\ldots,N$,
		\begin{align*}
			&a_j^{k,k'} := c_j^{k,k'} +\int ^{1}_{0} A_{k,k',z_{j}}(\Delta)\,\mathrm{d} \Delta,
			&&b_j^{k,k'} := c_j^{k,k'} + \frac{1}{N}\sum ^{N}_{m =1} A_{k,k',z_{j}}( \mu_{m}), \\
			&c_j^{k,k'} := \frac{1}{2} B_{2}\bigg(\bigg\{\frac{(k-k')z_{j}}{N}\bigg\}\bigg),
			&&\mu_m := \frac{2m-1}{2N}.
		\end{align*}
		Since $|B_2(x)|\le 1/6$ for all $x\in [0,1]$ and $|(x-1/2)(y-1/2)|\le 1/4$
		for all $x,y\in [0,1)$, we have trivially $|a_j^{k,k'}|\le 1/3$ and
		$|b_j^{k,k'}|\le 1/3$. It follows by induction that $|\prod_{j\in\setu} a_{j}^{k,k'}
		-\prod_{j\in\setu} b_{j}^{k,k'}|
		\le
		\big(\frac13\big)^{|\mathfrak{u} |-1}
		\sum _{j\in \mathfrak{u}} |a_{j}^{k,k'} -b_{j}^{k,k'} |$.
		
		We therefore consider the difference
		\begin{align*}
			a_{j}^{k,k'} -b_{j}^{k,k'}
			&= \int_0^1 A_{k,k',z_{j}}(\Delta)\,\rd\Delta
			- \frac{1}{N} \sum_{m=1}^N A_{k,k',z_{j}}(\mu_m) \\
			&= \sum^{N}_{m=1} \bigg(
			\int^{\frac{m}{N}}_{\frac{m-1}{N}} A_{k,k',z_{j}}(\Delta)\,\rd\Delta
			- \frac{1}{N} A_{k,k',z_{j}}\bigg(\frac{2m-1}{2N}\bigg) \bigg),
		\end{align*}
		which is precisely the error of a \emph{composite midpoint rule}
		approximation to the integral of $
		A_{k,k',z_{j}}(\Delta) = \{\frac{kz_j}{N} +\Delta \}\{\frac{k'z_j}{N} +\Delta \}
		- \frac{1}{2}\{\frac{kz_j}{N} +\Delta \} - \frac{1}{2}\{\frac{k'z_j}{N} +\Delta \}
		+ \frac{1}{4}
		$.
		
		Since $\frac{kz_j}{N}$ is a multiple of $\frac{1}{N}$, the function
		$\{\frac{kz_j}{N} +\Delta \}$ is \emph{linear} in $\Delta$ on each
		subinterval $[\frac{m-1}{N},\frac{m}{N})$ of length $\frac{1}{N}$, and so
		the midpoint rule is exact on each subinterval. The same conclusion holds
		for $\{\frac{k'z_j}{N} +\Delta \}$.
		
		On the other hand, the function $\{\frac{kz_j}{N} +\Delta
		\}\{\frac{k'z_j}{N} +\Delta \}$ is \emph{quadratic} in $\Delta$ on each
		subinterval $[\frac{m-1}{N},\frac{m}{N})$, and its second derivative is
		the constant function~$2$, which is uniformly continuous on
		$(\frac{m-1}{N},\frac{m}{N})$ and can be uniquely extended to
		$[\frac{m-1}{N},\frac{m}{N}]$. Therefore, the midpoint rule has error
		bounded by $\frac{2}{24}\frac{1}{N^3}$ on each subinterval, leading to the
		total error $|a_{j}^{k,k'} -b_{j}^{k,k'}|\le \frac{1}{12N^2}$, and in turn
		yielding $|\prod_{j\in\setu} a_{j}^{k,k'} -\prod_{j\in\setu} b_{j}^{k,k'}|
		\le \big(\frac13\big)^{|\mathfrak{u} |-1} \frac{|\setu|}{12N^2}$. This
		completes the proof.
	\end{proof}

	\section{CBC for shift algorithm}\label{sec:CBC-for-shifts}
	
	Theorem~\ref{thm} given in the previous section provides a good motivation to consider the following algorithm.
	\begin{algorithm*}
		\caption{CBC for shift}
		\begin{algorithmic}
			\INPUT $s_{\mathrm{max}}$, $N$, and  ${z}^{*}_{1} ,\dotsc {z}^{*}_{s_{\mathrm{max}}}$, a generating vector obtained by the CBC construction for randomly shifted lattice rules.
			\OUTPUT  shifts $\Delta ^{*}_{1} ,\dotsc ,\Delta
			^{*}_{s_{\mathrm{max}}} \in {S_N}$, and
			\[
			\kappa(N,s)=
			{\frac{e_{N,s}((z_1^*,\ldots,z_s^*),(\Delta_1^*,\ldots,\Delta_s^*))}
				{e_{N,s}^{\rm{sh}}(z_1^*,\ldots,z_s^*)}},
			\qquad s=1,\dots,s_{\max}.\]
			\STATE\textbf{do}
			\[
			\Delta ^{*}_{1} \in \mathrm{argmin}\big\{e^{2}_{N,1}({z}^{*}_{1} {,}\Delta _{1}) \mid \Delta _{1} \in {S_N}\big\},
			\]
			and
			$\kappa(N,1)={e_{N,1}(z_1^*,{\Delta}_1^*)}/{e_{N,1}^{\rm{sh}}(z_1^*)}$
			\FOR{$s$ from $2$ to $s_{\mathrm{max}}$}
			\STATE
			\[\Delta ^{*}_{s} \in
			\mathrm{argmin}\big\{e_{N,s}^{2} {(({z}^{*}_{1} ,\dotsc ,{z}^{*}_{s}),(\Delta ^{*}_{1} ,\dotsc ,\Delta ^{*}_{s-1} ,\Delta _{s}))} \mid \Delta _{s}
			\in {S_N}\big\},
			\]
			and $\kappa(N,s)=
			{e_{N,s}((z_1^*,\ldots,z_s^*),(\Delta_1^*,\ldots,\Delta_s^*))/
				e_{N,s}^{\rm{sh}}(z_1^*,\ldots,z_s^*)}$
			\ENDFOR
		\end{algorithmic}
	\end{algorithm*}
	
	\section{Numerical results}
	
	We ran the CBC for shift algorithm in weighted unanchored Sobolev
	spaces with product weights $\gamma_j = 1/j^2$, $\gamma_j = 0.9^j$,
	$\gamma_j = 0.75^j$, and $\gamma_j = 0.5^j$, with the number of points
	$N=1024$ and $2048$. We used the lattice generating vectors $\bsz^*$ from
	\cite{Kuo.Y.F_web_lattice}.
	
	Table~\ref{table:1/jsq} shows the values of the indices {$m_s^*$} for {the
		components of the shifts $\Delta_s^* = (2m_s^*-1)/(2N)$ together with the
		values of $\kappa(s,N)$, for the case $N=2048$ and $\gamma_j=1/j^2$. As} a
	comparison, we provide {also the} values of the ratio~\eqref{eq:ratio}
	with $\bsDelta^*$ replaced by the zero shift vector, denoting the new
	ratio  by $\kappa_0(N,s)$. We see that $\kappa(s,N)$ is {less than}
	$1$, whereas $\kappa_0(s,N)$ exceeds $1$.
	
	Table~\ref{table:0.5-to-j} shows the same values {for the case}
	$\gamma_j=0.5^j$. Again, we see that $\kappa(s,N)$ is {less than} $1$,
	whereas $\kappa_0(s,N)$ exceeds $1$. {The same observation holds for
		the other cases that we considered.}
	\begin{table}[H]
		\small
		\centering
		\begin{tabular}{|c|c|c|c|}
			\hline
			$s$ & ${m_s^*}$  & $\kappa(s,2048)$ & $\kappa_0(s,2048)$\\
			\hline
			1 & 1 & 0.708211 & 1.414765 \\
			\hline
			2 & 227 & 0.774829 & 1.242600 \\
			\hline
			3 & 17 & 0.804685 & 1.184076 \\
			\hline
			4 & 1955 & 0.817572 & 1.159945 \\
			\hline
			5 & 1273 & 0.827628 & 1.164227 \\
			\hline
			6 & 1250 & 0.835811 & 1.153188 \\
			\hline
			7 & 1698 & 0.841416 & 1.140436 \\
			\hline
			8 & 1970 & 0.845575 & 1.135741 \\
			\hline
			9 & 476 & 0.847988 & 1.134190 \\
			\hline
			10 & 646 & 0.850682 & 1.130419 \\
			\hline
			11 & 779 & 0.853486 & 1.129294 \\
			\hline
			12 & 1093 & 0.855818 & 1.126409 \\
			\hline
			13 & 1498 & 0.857239 & 1.123408 \\
			\hline
			14 & 550 & 0.859090 & 1.122288 \\
			\hline
			15 & 1218 & 0.860315 & 1.123017 \\
			\hline
			16 & 1124 & 0.861422 & 1.121392 \\
			\hline
			17 & 135 & 0.862420 & 1.120642 \\
			\hline
			18 & 717 & 0.863531 & 1.120035 \\
			\hline
			19 & 854 & 0.864463 & 1.119229 \\
			\hline
			20 & 1634 & 0.865152 & 1.118282 \\
			\hline
			21 & 1692 & 0.865776 & 1.117763 \\
			\hline
			22 & 1002 & 0.866488 & 1.116437 \\
			\hline
			23 & 1034 & 0.866953 & 1.117111 \\
			\hline
			24 & 249 & 0.867514 & 1.117083 \\
			\hline
			25 & 1477 & 0.868110 & 1.116343 \\
			\hline                                              \end{tabular}
		\,
		\begin{tabular}{|c|c|c|c|}
			\hline
			$s$ & ${m_s^*}$  & $\kappa(s,2048)$ & $\kappa_0(s,2048)$\\
			\hline
			26 & 626 & 0.868553 & 1.117042 \\
			\hline
			27 & 1987 & 0.869105 & 1.116248 \\
			\hline
			28 & 1676 & 0.869585 & 1.116454 \\
			\hline
			29 & 1323 & 0.869814 & 1.116148 \\
			\hline
			30 & 1037 & 0.870236 & 1.115625 \\
			\hline
			31 & 416 & 0.870557 & 1.116051 \\
			\hline
			32 & 416 & 0.870557 & 1.116332 \\
			\hline
			33 & 928 & 0.870846 & 1.116119 \\
			\hline
			34 & 928 & 0.870846 & 1.116082 \\
			\hline
			35 & 711 & 0.871200 & 1.115653 \\
			\hline
			36 & 711 & 0.871200 & 1.115341 \\
			\hline
			37 & 1852 & 0.871508 & 1.115202 \\
			\hline
			38 & 1852 & 0.871508 & 1.115458 \\
			\hline
			39 & 785 & 0.871817 & 1.115148 \\
			\hline
			40 & 785 & 0.871817 & 1.115281 \\
			\hline
			41 & 696 & 0.872074 & 1.115050 \\
			\hline
			42 & 1497 & 0.875772 & 1.114796 \\
			\hline
			43 & 1587 & 0.875981 & 1.114618 \\
			\hline
			44 & 638 & 0.876184 & 1.114463 \\
			\hline
			45 & 848 & 0.876350 & 1.114147 \\
			\hline
			46 & 954 & 0.876547 & 1.113879 \\
			\hline
			47 & 1042 & 0.876704 & 1.113629 \\
			\hline
			48 & 20 & 0.876866 & 1.113642 \\
			\hline
			49 & 589 & 0.876988 & 1.113757 \\
			\hline
			50 & 617 & 0.877128 & 1.113769 \\
			\hline
		\end{tabular}
		\caption{Shifts ${\Delta_s^*=(2m_s^*-1)/(2N)}$ and $\kappa(s,N)$
			obtained by the CBC for shift algorithm for $N=2048$, $\gamma_j=1/j^2$;
			and $\kappa_0(s,N)$, the values of $\kappa(s,N)$ {corresponding to zero
				shift}, for $s=1,\dots,50$. We see that $\kappa(s,N)$ is
			{less than} $1$.}
		\label{table:1/jsq}
	\end{table}
	\begin{table}[H]                                      \small
		\centering
		\begin{tabular}{|c|c|c|c|}
			\hline
			s & {$m_s^*$} & $\kappa(s,2048)$ & $\kappa_0(s,2048)$ \\
			\hline
			1 & 1 & 0.708211 & 1.414765 \\
			\hline
			2 & 227 & 0.774829 & 1.242600 \\
			\hline
			3 & 17 & 0.804685 & 1.184076 \\
			\hline
			4 & 1955 & 0.817572 & 1.159945 \\
			\hline
			5 & 422 & 0.829080 & 1.146448 \\
			\hline
			6 & 1698 & 0.836308 & 1.130689 \\
			\hline
			7 & 1917 & 0.841847 & 1.131901 \\
			\hline
			8 & 2005 & 0.845626 & 1.127075 \\
			\hline
			9 & 5 & 0.848370 & 1.121351 \\
			\hline
			10 & 135 & 0.851825 & 1.116117 \\
			\hline
			11 & 1139 & 0.853881 & 1.118105 \\
			\hline
			12 & 1410 & 0.857111 & 1.111780 \\
			\hline
			13 & 982 & 0.859260 & 1.109848 \\
			\hline
			14 & 1151 & 0.860542 & 1.107647 \\
			\hline
			15 & 751 & 0.862129 & 1.104858 \\
			\hline
			16 & 1043 & 0.863598 & 1.102854 \\
			\hline
			17 & 1083 & 0.864755 & 1.107635 \\
			\hline
			18 & 412 & 0.866106 & 1.107113 \\
			\hline
			19 & 211 & 0.867072 & 1.106381 \\
			\hline
			20 & 854 & 0.867879 & 1.105469 \\
			\hline
			21 & 418 & 0.868589 & 1.136709 \\
			\hline
			22 & 849 & 0.869226 & 1.164846 \\
			\hline
			23 & 13 & 0.876862 & 1.197944 \\
			\hline
			24 & 1280 & 0.877108 & 1.197713 \\
			\hline
			25 & 1229 & 0.882508 & 1.217421 \\
			\hline
		\end{tabular}
		\,
		\begin{tabular}{|c|c|c|c|}
			\hline
			s	& {$m_s^*$}  & $\kappa(s,2048)$ & $\kappa_0(s,2048)$ \\
			\hline
			26 & 11 & 0.890243 & 1.237195 \\
			\hline
			27 & 1696 & 0.896959 & 1.253651 \\
			\hline
			28 & 820 & 0.896507 & 1.256813 \\
			\hline
			29 & 1629 & 0.900509 & 1.269280 \\
			\hline
			30 & 1272 & 0.904119 & 1.279908 \\
			\hline
			31 & 1661 & 0.904847 & 1.283008 \\
			\hline
			32 & 633 & 0.909137 & 1.291162 \\
			\hline
			33 & 205 & 0.912924 & 1.298611 \\
			\hline
			34 & 1841 & 0.916205 & 1.305355 \\
			\hline
			35 & 2038 & 0.917085 & 1.307455 \\
			\hline
			36 & 1433 & 0.919874 & 1.312978 \\
			\hline
			37 & 405 & 0.920440 & 1.314927 \\
			\hline
			38 & 1042 & 0.921457 & 1.316980 \\
			\hline
			39 & 589 & 0.922440 & 1.319085 \\
			\hline
			40 & 1068 & 0.924557 & 1.322855 \\
			\hline
			41 & 1763 & 0.927052 & 1.326324 \\
			\hline
			42 & 1364 & 0.929271 & 1.329515 \\
			\hline
			43 & 1946 & 0.931415 & 1.332454 \\
			\hline
			44 & 214 & 0.932025 & 1.333719 \\
			\hline
			45 & 1511 & 0.933809 & 1.336227 \\
			\hline
			46 & 1835 & 0.934393 & 1.337355 \\
			\hline
			47 & 128 & 0.935941 & 1.339516 \\
			\hline
			48 & 1500 & 0.936470 & 1.340523 \\
			\hline
			49 & 1023 & 0.937853 & 1.342399 \\
			\hline
			50 & 561 & 0.939113 & 1.344153 \\
			\hline
		\end{tabular}
		\caption{Shifts ${\Delta_s^*=(2m_s^*-1)/(2N)}$ and $\kappa(s,N)$
			obtained by the CBC for shift algorithm for $N=2048$, $\gamma_j=0.5^j$;
			and $\kappa_0(s,N)$, the values of $\kappa(s,N)$ {corresponding to zero
				shift}, for $s=1,\dots,50$. We see that $\kappa(s,N)$ is
			{less than} $1$.} \label{table:0.5-to-j}
	\end{table}

	\paragraph{Acknowledgements}
	We gratefully acknowledge the financial support from the Australian
	Research Council (DP180101356).

\printbibliography

\Addressesone

\Addressestwo

\Addressesthree
\end{document}